\documentclass[11pt]{article}
\usepackage{amssymb}
\usepackage{amsmath}
\usepackage{amsfonts}
\usepackage{amsthm}
\usepackage{enumerate}
\input amssym.def

\title{\bf Asymptotic behavior for dissipative Korteweg-de Vrie equations}
\author{St\'ephane Vento, \\Universit\'e Paris-Est, \\Laboratoire d'Analyse
et de Math\'ematiques Appliqu\'ees,\\ 5 bd. Descartes, Cit\'e
Descartes, Champs-Sur-Marne,\\ 77454 Marne-La-Vall\'ee Cedex 2,
France}

\date{E-mail:\, stephane.vento@univ-paris-est.fr}

\numberwithin{equation}{section}

\newtheorem{theorem}{Theorem}[section]
\newtheorem{lemma}{Lemma}[section]

\newtheorem{corollary}{Corollary}[section]

\newtheorem{remark}{Remark}[section]

\def\R{\mathbb{R}}
\def\N{\mathbb{N}}

\def\S{\mathcal{S}}
\def\M{\mathcal{M}}

\def\eps{\varepsilon}

\def\sgn{\mathop{\rm sgn}\nolimits}
\def\mes{\mathop{\rm mes}\nolimits}
\def\intr{\int_{-\infty}^\infty}

\begin{document}
\maketitle

\noindent {\bf Abstract.}\, We study the large time behavior of
solutions to the dissipative Korteweg-de Vrie equations
$u_t+u_{xxx}+|D|^{\alpha}u+uu_x=0$ with $0<\alpha<2$. We find $v$
such that $u-v$ decays like $t^{-r(\alpha)}$ as
$t\rightarrow\infty$ in various Sobolev
norm.\\

\noindent {\bf Keywords :} KdV-like equations, dissipative dispersive equations, large time behavior\\
{\bf AMS Classification :} 35Q53, 35B40

\section{Introduction}

In this paper we study the asymptotic behavior of solutions to the
following dissipative KdV equations
\begin{equation}\tag{dKdV}\label{eq}\left\{\begin{array}{ll}u_t+u_{xxx}+|D|^\alpha u+uu_x=0, & t\in\R_+, x\in\R,\\
u(0,x)=u_0(x), & x\in\R,\end{array}\right.\end{equation}
with $0<\alpha<2$ and where $|D|^\alpha$ is the L\'{e}vy operator defined through its Fourier transform by
$\widehat{|D|^\alpha\varphi}(\xi)=|\xi|^\alpha\widehat{\varphi}(\xi)$. Here $u=u(t,x)$ is a real-valued function.

The (\ref{eq}) equations are dissipative versions of the
well-known KdV equation
\begin{equation}\label{kdv}u_t+u_{xxx}+uu_x=0\end{equation} which
have been extensively studied. Equation (\ref{kdv}) is completely
integrable and there exists an infinite sequence of conserved
quantities. For sufficiently smooth initial data, we know that
global in time solutions exist and can be asymptotically written
as a sum of traveling wave solutions, called solitons, see
\cite{MR874343}, \cite{MR0404890}.

\vskip 0.5cm

Concerning the pure dissipative equation
\begin{equation}\label{fbe}u_t+|D|^\alpha u+uu_x=0,\end{equation} it has been proposed to model a variety
of physical phenomena, such that the growth of molecular
interfaces (cf. \cite{PhysRevLett.56.889}). Also, in
\cite{MR2158256}, Jourdain, M{\'e}l{\'e}ard and Woyczynski pointed
out the main interest of equation (\ref{fbe})  in probability
theory. Biler, Funaki and Woyczynski proved in \cite{MR1637513}
several local and global well-posedness results, in particular in
the general setting $0<\alpha\leq 2$, they obtained weak solutions
of (\ref{fbe}). Using the Fourier splitting method first
introduced by Schonbek in \cite{MR571048}, they showed that
regular solutions satisfy the estimate
\begin{equation}\label{estl2}\|u(t)\|_{L^2}\leq
c(1+t)^{-1/2\alpha}\end{equation} for all $t>0$. This result was
improved by Biler, Karch and Woyczynski \cite{MR1708995} in the
case of a diffusion operator of the form
$-\partial_x^2+|D|^\alpha$. See also \cite{karch-2007} for
asymptotic results concerning (\ref{fbe}) with $1<\alpha<2$.

\vskip 0.5cm

Let us turn back to the (\ref{eq}) equation. The Cauchy problem
(\ref{eq}) with $0\leq\alpha\leq 2$ has been shown to be globally
well-posed in the Sobolev spaces $H^s(\R)$ for all $s>-3/4$ and
furthermore, the solution $u(t)$ belongs to $H^\infty(\R)$ for any
$t>0$ (cf. \cite{MR1889080}). When $\alpha=1/2$, (\ref{eq}) models
the evolution of the free surface for shallow water waves damped
by viscosity, see \cite{ott:1432}. When $\alpha=2$, (\ref{eq}) is
the so-called KdV-Burgers equation which models the propagation of
weakly nonlinear dispersive long waves in some contexts when
dissipative effects occur (see \cite{ott:1432}).  In the case
$\alpha=0$, (\ref{eq}) reads
\begin{equation}\label{alpha0}u_t+u_{xxx}+u+uu_x=0\end{equation}
and it is easy to get the decay rate for the $L^2$-norm of the
solution. Indeed, multiplying (\ref{alpha0}) by $u$ and
integrating over $\R$ give for regular solutions the equality
$$\frac 12\partial_t\intr u^2(t,x)dx+\intr u^2(t,x)dx=0,$$ and it follows
immediately that $$\|u(t)\|_{L^2}= O(e^{-t}) \textrm{ as }
t\rightarrow \infty.$$ Now consider the KdV-Burgers equation
((\ref{eq}) with $\alpha=2$). In a sharp contrast with what occurs
for (\ref{alpha0}), Amick, Bona and Schonbek \cite{MR1012198}
proved that if $u_0\in L^1(\R)\cap H^2(\R)$, then the
corresponding solution satisfies
\begin{equation}\label{estkdvb}\|u(t)\|_{L^2}\leq c(1+t)^{-1/4}\end{equation} and furthermore, this estimate is optimal
 for a generic class of functions.
The proof of this result is based on a subtle use of the Hopf-Cole
transformation. Later, Karch \cite{MR1643236} improved this result
by showing that the asymptotic profile of the solution with a mass
$M$ is given by the fundamental solution $U_M$ of the viscous
Burgers equation (eq. (\ref{fbe}) with $\alpha=2$)
$$u_t-u_{xx}+uu_x=0$$ with the same mass. More precisely, we have
$$t^{(1-1/p)/2}\|u(t)-U_M(t)\|_{L^p}\rightarrow 0\quad\textrm{as}\ t\rightarrow \infty$$ for each
$p\in[1,\infty]$. In other words, we can say that for large times,
the dispersion is negligible compared to dissipation and
nonlinearity effects. His method of proof is based on a scaling
argument. This kind of behavior was also heuristically observed by
Dix in \cite{MR1187625}. He called this situation the "balanced
case" because both dissipation and nonlinearity contributions
appear in the long time behavior of the solution, this is formally
expressed by the relation $\alpha=2$.

In the present paper we study the so-called "asymptotically weak
nonlinearity case" $\alpha<2$. For a large class of equations,
solution of the nonlinear problem asymptotically looks like
solution of the corresponding linear problem (with same initial
data). One of the goals of this article is to show that similar
behaviors occur for (\ref{eq}) with $0<\alpha <2$.\\

 Following the
works of Karch \cite{MR1727212}, we shall mainly work on the
integral formulation of (\ref{eq}) :
\begin{equation}\label{duhamel}u(t)=S_\alpha(t)\ast u_0-\frac 12\int_0^tS_\alpha(t-s)\ast\partial_x u^2(s)ds\end{equation}
valid for any sufficiently regular solution, and where $S_\alpha(t)$ is defined by $$S_\alpha(t,x)=
\frac{1}{2\pi}\intr e^{ix\xi}e^{(i\xi^3-|\xi|^\alpha)t}d\xi,\quad t>0.$$
First, using the properties of the generalized heat kernel, we
give a complete asymptotic expansion of the free solution
$S_\alpha(t)\ast u_0$. After deriving the decay rates estimates of
the solution in various Sobolev norms $\|\cdot\|$, we show that
$\|u(t)-S_\alpha(t)\ast u_0\|$ is bounded by $ct^{-r(\alpha)}$,
$r(\alpha)>0$. Next, we improve this result by finding terms
$w=w(t,x)$ such that $\|u(t)-S_\alpha(t)\ast u_0-w(t)\|$ decays to
zero faster than $t^{-r(\alpha)}$. \vskip 0.5cm

\noindent{\bf Notation.} The notation to be used are standard. The letter $c$ denotes a constant which may change at each occurrence.
For $p\in[1,\infty]$ we define the Lebesgue space $L^p(\R)$ by its norm $\|f\|_{L^p}=\Big(\intr |f(x)|^pdx\Big)^{1/p}$ with the usual modification
for $p=\infty$. If $f=f(t,x)$ is a space-time function, the $L^p$-norm of $f$ will be taken in the $x$-variable. For $j\geq 0$ and $p\in[1,\infty]$,
the Sobolev spaces $H^{p,j}(\R)$ and $\dot{H}^{p,j}(\R)$ are respectively endowed with the norms $\|f\|_{H^{p,j}}=\|f\|_{L^p}+\|\partial_x^jf\|_{L^p}$
and $\|f\|_{\dot{H}^{p,j}}=\|\partial_x^jf\|_{L^p}$. When $p=2$, we simplify by the notation $H^j(\R)$ and $\dot{H}^j(\R)$. If $f\in\S'(\R)$,
we define its Fourier transform by setting $\hat{f}(\xi)=\mathcal{F}f(\xi)=\intr e^{-ix\xi}f(x)dx$.\\
We introduce $G_\alpha$, the fundamental solution of the equation $u_t+|D|^\alpha u=0$, i.e.
$$G_\alpha(t,x)=\frac{1}{2\pi}\intr e^{ix\xi}e^{-t|\xi|^\alpha}d\xi,\quad t>0.$$
It is clear that $G_\alpha$ has the self-similarity property
\begin{equation}\label{autosi}G_\alpha(t,x) =
t^{-1/\alpha}G_\alpha(1,xt^{-1/\alpha}),\quad x\in\R,
t>0.\end{equation} On the other hand, we know that $G_\alpha(t)\in
H^{p,j}(\R)$ for any $p\in[1,\infty]$ and $j\geq 0$, see for
instance \cite{miao-2006}.\\ Finally, for $f\in L^1(x^jdx)$,
$j\in\N$, we set $\M_j(f)=\intr f(x)x^jdx$.

\section{Main results}
\label{sec-main} As we are going to show, the solution of
(\ref{eq}) can be approximated by the solution of the
corresponding linear equation. We first give a complete asymptotic
expansion of $S_\alpha(t)\ast u_0$, which will be used in the
proof of the main theorem.
\begin{theorem}\label{complin} Let $p\in[1,\infty]$ and $j,N\in\N$. Then for all $t\geq 1$ and $u_0\in L^1((1+|x|)^{N+1}dx)$,
\begin{multline}\Big\|S_\alpha(t)\ast u_0-\sum_{n=0}^N\frac{(-1)^n}{n!}\M_n(u_0)\partial_x^nG_\alpha(t)-\sum_{k=1}^N\frac{t^k}{k!}
\sum_{\ell=0}^{N-1}\frac{(-1)^\ell}{\ell !}\M_\ell(u_0)\partial_x^\ell(-\partial_x)^{3k}G_\alpha(t)\Big\|_{\dot{H}^{p,j}}\\
\leq ct^{-(1-1/p)/\alpha-j/\alpha-(N+1)/\alpha}\end{multline}
\end{theorem}
\begin{remark} When $N=0$, the sum $\sum_{k=1}^N$ in (\ref{complin}) has to be understood as 0, and thus (\ref{complin}) reads
\begin{equation}\label{estSGlin}\|S_\alpha(t)\ast u_0-\M_0(u_0)G_\alpha(t)\|_{\dot{H}^{p,j}}\leq ct^{-(1-1/p)/\alpha-j/\alpha-1/\alpha}.
\end{equation}
If $N=1$, we have the following asymptotic expansion for $S_\alpha(t)\ast u_0$,
\begin{multline*}\|S_\alpha(t)\ast u_0-\M_0(u_0)G_\alpha(t)+\M_1(u_0)\partial_xG_\alpha(t)+t\M_0(u_0)\partial_x^3G_\alpha(t)\|_{\dot{H}^{p,j}}\\
\leq ct^{-(1-1/p)/\alpha-j/\alpha-2/\alpha}.\end{multline*}
\end{remark}

\begin{remark} The term $\sum_{n=0}^N\frac{(-1)^n}{n!}\M_n(u_0)\partial_x^nG_\alpha(t)$ in (\ref{complin}) corresponds to the asymptotic expansion
of $G_\alpha(t)\ast u_0$, solution to
 the generalized heat equation $u_t+|D|^\alpha u=0$. The other terms are due to the dispersive effects and appear only for $N\geq 1$.
\end{remark}

\vskip 0.5cm

Now we consider the nonlinear equation (\ref{eq}) with
$0<\alpha<2$. Throughout this paper, we make the following
assumptions :
\begin{equation}\label{as1}u_0\in L^1(\R)\cap L^2(\R),\end{equation}
\begin{equation}\label{as2}u\in C(]0,\infty[;H^\infty(\R)),\end{equation}
\begin{equation}\label{as3}\textrm{if}\ u_0\in H^{j}(\R),\ \textrm{then}\ \sup_{t\geq 0}\|\partial_x^ju(t)\|_{L^2}<\infty.\end{equation}
For $u_0\in L^2(\R)$, existence of global solutions satisfying
(\ref{as2}) was proved for example in \cite{MR1889080}. Moreover,
if $u_0\in H^j(\R)$, it was shown that the solution is continuous
from $[0,\infty[$ to $H^j(\R)$. In Section \ref{sec-bound}, we
will show that assumption (\ref{as3}) is verified for such
solutions when $u_0\in L^1(\R)\cap L^\infty(\R)$, at least in the
case $\alpha>1$.

\begin{theorem}\label{th-decayrate} Let $p\in[2,\infty]$ and $j\in \N$. Assume that $u_0\in H^{j+1}(\R)\cap L^1(\R)$ and (\ref{as2})-(\ref{as3})
hold true. Then we have \begin{equation}\label{decayrate}\|u(t)\|_{\dot{H}^{p,j}}\leq c(1+t)^{-(1-1/p)/\alpha-j/\alpha},\quad t>0.\end{equation}
When $j=0$, (\ref{decayrate}) is valid for all $p\in[1,\infty]$.
\end{theorem}

\vskip 0.5cm

Next we find the first term in the asymptotic expansion of the solution.
\begin{theorem}\label{th-ordre1}Let $p\in[2,\infty]$ and $j\in\N$. We assume that $u_0\in H^{j+3}(\R)\cap L^1(\R)$ and that the solution $u$
satisfies (\ref{as2})-(\ref{as3}). Then, for all $t>0$,
$$\|u(t)-S_\alpha(t)\ast u_0\|_{\dot{H}^{p,j}}\leq c\left\{\begin{array}{lll}(1+t)^{(-(1-1/p)/\alpha-j/\alpha)-1/\alpha}
&\textrm{for} & 0<\alpha<1,\\ (1+t)^{(-(1-1/p)-j)-1}\log(1+t) &\textrm{for} & \alpha=1,\\ (1+t)^{(-(1-1/p)/\alpha-j/\alpha)-(2/\alpha-1)}
 &\textrm{for} & 1<\alpha<2.\end{array}\right.$$
\end{theorem}

In view of Theorems \ref{th-decayrate} and \ref{th-ordre1}, it is
clear that decay rate of $u(t)-S_\alpha(t)\ast u_0$ in
$\dot{H}^{p,j}$-norm is better than when considering only $u(t)$.
In order to find other terms in the asymptotic expansion, we need
to consider separately the cases $0<\alpha<1$, $\alpha=1$ and
$1<\alpha<2$. \vskip 0.5cm

When $0<\alpha<1$ or $\alpha=1$, the difference between the asymptotic behavior of the first and second term is subtle.
For the first term, we have $\|u(t)-S_\alpha(t)\|_{\dot{H}^{p,j}}=O(t^{-(1-1/p)/\alpha-j/\alpha-1/\alpha})$ (when $\alpha<1$),
whereas for the second one, say $w(t)$, we have $\|u(t)-S_\alpha(t)-w(t)\|_{\dot{H}^{p,j}}=o(t^{-(1-1/p)/\alpha-j/\alpha-1/\alpha})$.
The following result holds for $\alpha\leq 1$.

\begin{theorem}\label{th-ordre2.1} Suppose $p\in[2,\infty]$, $j\in\N$, $u_0\in H^{j+3}(\R)\cap L^1(\R)$ and that (\ref{as2})-(\ref{as3}) are verified.
\begin{enumerate}[(i)]
\item If $0<\alpha<1$, then
\begin{equation}\label{asympa1}t^{((1-1/p)/\alpha+j/\alpha)+1/\alpha}\Big\|u(t)-S_\alpha(t)\ast u_0+\frac 12\Big(\int_0^\infty\intr u^2(s,y)dyds\Big)
\partial_xG_\alpha(t)\Big\|_{\dot{H}^{p,j}}\rightarrow 0\end{equation} as $t\rightarrow \infty$.
\item If $\alpha=1$, then
\begin{equation}\label{asympa2}\frac{t^{(1-1/p)+j+1}}{\log t}\Big\|u(t)-S_1(t)\ast u_0+\frac{M^2}{4\pi}(\log t)\partial_xG_1(t)\Big\|_{\dot{H}^{p,j}}\rightarrow 0
\end{equation} where $M=\M_0(u_0)=\intr u_0$.
\end{enumerate}
\end{theorem}

\begin{remark}In the case $\alpha<1$, the integral $\int_0^\infty\intr u^2(s,y)dyds$ which
 appears in (\ref{asympa1}) is convergent due to Theorem \ref{th-decayrate} :
$$\int_0^\infty\intr u^2(s,y)dyds=\int_0^\infty\|u(s)\|_{L^2}^2ds\leq c\int_0^\infty(1+s)^{-1/\alpha}ds<\infty.$$
\end{remark}
\vskip 0.5cm

Now we deal with the case $1<\alpha<2$. In this situation we get
an asymptotic expansion of the solution at the rate
$O(t^{-(1-1/p)/\alpha-j/\alpha-1/\alpha})$
 (in $\dot{H}^{p,j}$-norm, and for almost every $\alpha$) but we need more than two terms in this expansion to derive it.
The main idea is to use the successive terms $F^n(t)$
  which appear in the Picard iterative scheme applied to the Duhamel formulation
  (\ref{duhamel}), i.e.
$$\left\{\begin{array}{ll}F^0(t)=S_\alpha(t)\ast u_0,\\ F^{n+1}(t)=S_\alpha(t)\ast u_0-\frac 12\int_0^tS_\alpha(t-s)\ast \partial_x(F^n(s))^2ds.\end{array}\right.$$

\begin{theorem}\label{th-ordre2.2}Let $1<\alpha<2$, $p\in[2,\infty]$, $j\in\N$ and $u_0\in H^{j+3}(\R)\cap L^1(\R)$.
Suppose that conditions (\ref{as2}) and (\ref{as3}) are satisfied.
\begin{enumerate}[(i)]
\item If $\frac{2N+1}{N+1}<\alpha<\frac{2N+3}{N+2}$ for a $N\in \N$, then
$$\|u(t)-F^{N+1}(t)\|_{\dot{H}^{p,j}}\leq c(1+t)^{-(1-1/p)/\alpha-j/\alpha-1/\alpha}.$$
\item If $\alpha=\frac{2N+3}{N+2}$ for a $N\in\N$, then
$$\|u(t)-F^{N+1}(t)\|_{\dot{H}^{p,j}}\leq c(1+t)^{-(1-1/p)/\alpha-j/\alpha-1/\alpha}\log(1+t).$$
\end{enumerate}
\end{theorem}

\begin{remark}The results obtained in this paper for (\ref{eq}) could be certainly adapted to more general dispersive dissipative
equations taking the form \begin{equation}\label{geneq}u_t-|D|^r\partial_x u+|D|^\alpha u+\partial_x f(u)=0,\end{equation} where $f$
is sufficiently smooth function behaving like $u|u|^{q-1}$ at the origin.
Such general models were studied by Dix in \cite{MR1187625}. Similar asymptotic expansion for solutions to (\ref{geneq}) could be
obtained in certain cases, when dissipation is not negligible in comparison with dispersion and nonlinearity :
$$\left\{\begin{array}{ll}\alpha\leq r+1,\\ 0<\alpha< q.\end{array}\right.$$
\end{remark}

The remainder of this paper is organized as follows. In Section
\ref{sec-estlin}, we derive linear estimates and prove Theorem
\ref{complin}. Uniform estimates of the nonlinear solution are
obtained in Section \ref{sec-bound}.
 The decay rate (\ref{decayrate}) is established in Section \ref{sec-decay}.
 Finally, Section \ref{sec-higher} is devoted to the proof of Theorems \ref{th-ordre1}, \ref{th-ordre2.1} and \ref{th-ordre2.2}.

\section{Linear estimates}\label{sec-estlin}
In this section, we prove some estimates related with $S_\alpha(t)$ and $G_\alpha(t)$. Our first lemma is a direct consequence of the self-similarity of $G_\alpha$.

\begin{lemma}
For any $p\in[1,\infty]$ and $j\in\N$,
\begin{equation}\label{estGt}\|G_\alpha(t)\|_{\dot{H}^{p,j}}
= ct^{-(1-1/p)/\alpha-j/\alpha}.\end{equation}
\end{lemma}
\begin{proof} Equality (\ref{autosi}) and a change of variables yield
\begin{align*}\|G_\alpha(t)\|_{\dot{H}^{p,j}} &=\Big(\intr t^{-(j+1)p/\alpha}|\partial_x^jG_\alpha(1,xt^{-1/\alpha})|^pdx\Big)^{1/p}\\
& =t^{-(j+1)/\alpha} t^{1/\alpha p}\Big(\intr|\partial_x^jG_\alpha(1,y)|^pdy\Big)^{1/p}.
\end{align*}
The case $p=\infty$ is straightforward.
\end{proof}

Let us recall the following elementary result which will be
extensively used in our future considerations. A proof of
(\ref{ele2}) can be found in \cite{MR1455428}.
\begin{lemma}\label{elelem}
If $1\leq k\leq j$ and $f\in H^j(\R)$, then \begin{equation}\label{ele1}\|f\|_{L^\infty}^2\leq \|f\|_{L^2}\|f_x\|_{L^2},\quad\textrm{and}\quad
\|\partial_x^kf\|_{L^2}\leq \|f\|_{L^2}^{1-k/j}\|\partial_x^jf\|_{L^2}^{k/j}.\end{equation}
Moreover, for any $f\in L^2((1+|x|)dx)$, one has
\begin{equation}\label{ele2}\|f\|_{L^1}^2\leq c\|f\|_{L^2}\|\partial_\xi\widehat{f}\|_{L^2}.
\end{equation}
\end{lemma}

Next lemma describes the asymptotic behavior of $S_\alpha(t)$.

\begin{lemma}For any $p\in[1,\infty]$ and $j,N\in\N$,
\begin{equation}\label{estSG}\Big\|S_\alpha(t)-\sum_{n=0}^N \frac{t^n}{n!}(-\partial_x)^{3n}G_\alpha(t)\Big\|_{\dot{H}^{p,j}}
\leq ct^{-(1-1/p)/\alpha-j/\alpha-(3/\alpha-1)(N+1)}.\end{equation}
\end{lemma}
\begin{proof}  Setting $A(t)=S_\alpha(t)-\sum_{n=0}^N\frac{t^n}{n!}(-\partial_x)^{3n}G_\alpha(t)$, we obtain
$$\mathcal{F}(\partial_x^jA(t))(\xi)=(i\xi)^je^{-t|\xi|^\alpha}\Big(e^{it\xi^3}-\sum_{n=0}^N\frac{t^n}{n!}(-i\xi)^{3n}\Big).$$
Using the Taylor expansion of the exponential function, we have
$$\Big|e^{it\xi^3}-\sum_{n=0}^N\frac{(it\xi^3)^n}{n!}\Big|\leq
\frac{(t|\xi|^3)^{N+1}}{(N+1)!}.$$  Thus, Plancherel theorem and
the change of variables $\xi=t^{-1/\alpha}\eta$ give
\begin{align*}\|\partial_x^j A(t)\|_{L^2}^2 &\leq c \intr
|\xi|^{2j}e^{-2t|\xi|^\alpha}(t|\xi|^3)^{2(N+1)}d\xi\\ &=
ct^{2(N+1)} \intr|\xi|^{2(j+3N+3)}e^{-2t|\xi|^\alpha}d\xi\\ &=
ct^{-1/\alpha-2j/\alpha-2(3/\alpha-1)(N+1)},\end{align*} which
yields the result for $p=2$. Now the case $p=\infty$ follows
immediately from (\ref{ele1}). When $p=1$, we use estimate
(\ref{ele2}). One has
\begin{align*} \|\partial_{\xi}\mathcal{F}(\partial_x^jA(t))\|_{L^2}& \leq c\Big(\intr \Big[|
j\xi^{j-1}(t\xi^3)^{N+1}|^2 +|t\xi^{j+\alpha-1}(t\xi^3)^{N+1}|^2
\\&\quad+
|\xi|^{2j}\Big|3it\xi^2e^{it\xi^3}-\sum_{n=0}^N\frac{3n(it)^n\xi^{3n-1}}{n!}\Big|^2\Big]
e^{-2t|\xi|^\alpha}d\xi\Big)^{1/2}\\ &\leq
ct^{N+1}\Big(\intr j|\xi|^{2(j-1+3(N+1))}e^{-2t|\xi|^\alpha}d\xi \Big)^{1/2}\\
&\quad
+ct^{N+2}\Big(\intr|\xi|^{2(j+\alpha-1+3(N+1))}e^{-2t|\xi|^\alpha}d\xi\Big)^{1/2}\\
&\quad
+ct\Big(\intr|\xi|^{2(j+2)}|t\xi^3|^{2N}e^{-2t|\xi|^\alpha}d\xi\Big)^{1/2}\\
&\leq ct^{-1/2\alpha-j/\alpha-(3/\alpha-1)(N+1)+1/\alpha}.
\end{align*}
It follows that (\ref{estSG}) holds true for $p=1$ and then for all $p\in[1,\infty]$ by interpolation.
\end{proof}

\begin{lemma}\label{lem-estlin}For all $p\in[2,\infty]$ and $j\in\N$,
\begin{equation}\label{estSLp}\|S_\alpha(t)\|_{\dot{H}^{p,j}}\leq ct^{-(1-1/p)/\alpha-j/\alpha}\end{equation}
and $$\|S_\alpha(t)\|_{\dot{H}^{1,j}}\leq ct^{-j/\alpha}(1+t^{1-3/\alpha}).$$
\end{lemma}
\begin{proof} For $p=2$,
$\|S_\alpha(t)\|_{\dot{H}^{j}}= \|G_\alpha(t)\|_{\dot{H}^{j}}=ct^{-1/2\alpha-j/\alpha}$. Then (\ref{estSLp}) follows by the first inequality in
(\ref{ele1}) and by interpolation. Concerning the $L^1$-norm, (\ref{estSG}) with $N=0$ and (\ref{estGt}) provide
$$\|S_\alpha(t)\|_{\dot{H}^{1,j}}\leq \|S_\alpha(t)-G_\alpha(t)\|_{\dot{H}^{1,j}}+\|G_\alpha(t)\|_{\dot{H}^{1,j}} \leq ct^{-j/\alpha}(1+t^{1-3/\alpha}).$$
\end{proof}

Now we state a decomposition lemma for convolution products.
\begin{lemma}\label{estconv} Let $p\in[1,\infty]$ and $N\in\N$. For any $h\in L^1((1+|x|)^{N+1}dx)$ and $g\in C^{N+1}(\R)\cap H^{p,N+1}(\R)$,
$$\Big\|g\ast h-\sum_{n=0}^N \frac{(-1)^n}{n!}\M_n(h)\partial_x^ng\Big\|_{L^p} \leq c\|\partial_x^{N+1}g\|_{L^p}\|h\|_{L^1(|x|^{N+1}dx)}.$$
\end{lemma}
\begin{proof} It is an easy consequence of the Taylor formula as well as Young inequality.\end{proof}

Applying Lemma \ref{estconv} with $g=\partial_x^jG_\alpha(t)$ and using estimate (\ref{estGt}), we deduce the
\begin{corollary}\label{estG} If $p\in[1,\infty]$ and $j,N\in\N$, then $$\Big\|G_\alpha(t)\ast h-\sum_{n=0}^N\frac{(-1)^n}{n!}
\M_n(h)\partial_x^nG_\alpha(t)\Big\|_{\dot{H}^{p,j}}\leq ct^{-(1-1/p)/\alpha-j/\alpha-(N+1)/\alpha}\|h\|_{L^1(|x|^{N+1}dx)}$$ for any $h\in L^1((1+|x|)^{N+1}dx)$.
\end{corollary}

We are now in a position to prove Theorem \ref{complin}.

\begin{proof}[Proof of Theorem \ref{complin}] By the triangle inequality,
\begin{eqnarray*}
&&  \Big\|S_\alpha(t)\ast u_0-\sum_{n=0}^N\frac{(-1)^n}{n!}\M_n(u_0)\partial_x^nG_\alpha(t)-\sum_{k=1}^N\frac{t^k}{k!}
\sum_{\ell=0}^{N-1}\frac{(-1)^\ell}{\ell !}\M_\ell(u_0)\partial_x^\ell(-\partial_x)^{3k}G_\alpha(t)\Big\|_{\dot{H}^{p,j}}\\
&&\quad \leq \Big\|S_\alpha(t)\ast u_0-G_\alpha(t)\ast u_0-\sum_{k=1}^N\frac{t^k}{k!}(-\partial_x)^{3k}G_\alpha(t)\ast
u_0\Big\|_{\dot{H}^{p,j}}\\& &\qquad+\Big\|G_\alpha(t)\ast u_0-\sum_{n=0}^N \frac{(-1)^n}{n!}\M_n(u_0)\partial_x^nG_\alpha(t)\Big\|_{\dot{H}^{p,j}}\\
 &&\qquad+\sum_{k=1}^N\frac{t^k}{k!}\Big\|(-\partial_x)^{3k}G_\alpha(t)\ast u_0-\sum_{\ell=0}^{N-1}\frac{(-1)^\ell}{\ell !}\M_\ell(u_0)
 \partial_x^\ell(-\partial_x)^{3k}G_\alpha(t)\Big\|_{\dot{H}^{p,j}}\\ &&\quad := I+II+III.
\end{eqnarray*}
$I$ is estimated with the help of (\ref{estSG}),
\begin{align*}I &=\Big\|\partial_x^j\Big(S_\alpha(t)-\sum_{k=0}^N\frac{t^k}{k!}(-\partial_x)^{3k}G_\alpha(t)\Big)\ast u_0\Big\|_{L^p} \\ &\leq
\Big\|\partial_x^j\Big(S_\alpha(t)-\sum_{k=0}^N\frac{t^k}{k!}(-\partial_x)^{3k}G_\alpha(t)\Big)\Big\|_{L^p}\|u_0\|_{L^1}\\
&\leq ct^{-(1-1/p)/\alpha-j/\alpha-(3/\alpha-1)(N+1)}\leq
ct^{-(1-1/p)/\alpha-j/\alpha-(N+1)/\alpha},\end{align*} since $\alpha<2$.
Concerning $II$, we use Corollary \ref{estG} as follows :
$$II\leq ct^{-(1-1/p)/\alpha-j/\alpha-(N+1)/\alpha}\|u_0\|_{L^1(|x|^{N+1}dx)}.$$
Finally for the term $III$, Corollary \ref{estG} allows us to conclude
\begin{align*}III &\leq \sum_{k=1}^{N}\frac{t^k}{k!}\Big\|\partial_x^{3k+j}\Big(G_\alpha(t)\ast u_0-\sum_{\ell=0}^{N-1}\frac{(-1)^\ell}{\ell !}
\M_\ell(u_0)\partial_x^\ell G_\alpha(t)\Big)\Big\|_{L^p}\\ &\leq ct^{-(1-1/p)/\alpha-j/\alpha-N/\alpha}\sum_{k=1}^Nt^{(1-3/\alpha)k}\\
&\leq ct^{-(1-1/p)/\alpha-j/\alpha-N/\alpha+1-3/\alpha}\\ &\leq ct^{-(1-1/p)/\alpha-j/\alpha-(N+1)/\alpha}.\end{align*}
\end{proof}

\section{Uniform estimates of solutions to (\ref{eq})}\label{sec-bound}
We begin by the proof of Theorem \ref{th-decayrate} in the case
$j=0$ and $p=1$.
\begin{lemma}\label{decayl1} Let $u_0\in L^1(\R)\cap L^2(\R)$ and $u$ be a solution of (\ref{eq}) satisfying (\ref{as2}). Then for all $t>0$,
$$\|u(t)\|_{L^1}\leq \|u_0\|_{L^1}.$$
\end{lemma}
\begin{proof} Multiply (\ref{eq}) by $\sgn u$ and then integrate over $\R$ :
\begin{equation}\label{aprioril1}\partial_t\|u(t)\|_{L^1}=-\intr (u_{xxx}+|D|^\alpha u+uu_x)\sgn u.\end{equation}
We are going to show that for each $t>0$, the right-hand side of
(\ref{aprioril1}) is negative. Note that assumption (\ref{as2})
means that for each $t>0$, there exists $c=c(t)$ such that
\begin{equation}\label{reas2}\forall j\geq 0,\ \|\partial_x^j
u(t)\|_{L^2}\leq c.
\end{equation}

Since $-|D|^\alpha$ is the generator of contraction semigroup in
$L^1(\R)$, for each $u\in\mathcal{D}(-|D|^\alpha)$ (the domain of
$-|D|^\alpha$),
\begin{align*}-\intr |D|^\alpha u\sgn u &=\lim_{s\rightarrow 0}\intr \frac{e^{-s|D|^\alpha}u-u}{s}\sgn u\\ &=
\lim_{s\rightarrow 0}\frac{1}{s}\intr\Big(e^{-s|D|^\alpha}u\sgn u-|u|\Big)\\ &\leq \limsup_{s\rightarrow 0} \frac 1s \Big(\intr |e^{-s|D|^\alpha}u|-\intr |u|\Big)\\ &\leq 0.
\end{align*}
This last inequality is sometimes called Kato inequality, see
\cite{MR531061}-\cite{MR1637513}. To show that the other terms in
the right-hand side of (\ref{aprioril1}) are also negative, we
introduce the following smooth regularization of the $\sgn$
function
$$\sgn_\eta(\xi)=\left\{\begin{array}{lll}1 &\rm{if}& \xi>\eta\pi/2,\\ \sin(\xi/\eta)&\rm{if}& |\xi|\leq \eta\pi/2,\\ -1 &\rm{if}& \xi<-\eta\pi/2.\end{array}\right.$$
Then, an integration by parts gives
\begin{align*}-\intr uu_x\sgn u  &=-\lim_{\eta\rightarrow 0}\intr uu_x\sgn_\eta u =\frac 12\lim_{\eta\rightarrow 0}\intr u^2u_x\sgn_\eta'u.
\end{align*}
On the other hand, $\sgn_\eta'$ has its support in
$[-\eta\pi/2,\eta\pi/2]$ and $|\sgn_\eta'|\leq 1/\eta$, hence
setting $M_\eta=\{x : |u|< \eta\pi/2, u_x\neq 0\}$, one has
$\mes(M_\eta)\rightarrow 0$ ($\mes$ denotes the Lebesgue measure)
and
$$\Big|\intr u^2u_x.\sgn_\eta'u\Big|\leq \frac 1\eta \int_{M_\eta}|u^2u_x| \leq c\eta\|u_x\|_{L^2}\Big(\int_{M_\eta}\Big)^{1/2}\rightarrow 0$$
as $\eta\rightarrow 0$ by (\ref{reas2}). Thus  $\intr uu_x\sgn u=
0$. We proceed similarly for the last term,
\begin{align*}-\intr u_{xxx}\sgn u &=-\lim_{\eta\rightarrow 0}\intr u_{xxx}\sgn_\eta u=\lim_{\eta\rightarrow 0}\intr u_{xx}u_x\sgn_\eta'u
\end{align*}
and $$\Big|\intr u_{xx}u_x\sgn_\eta'u\Big|\leq \frac
1\eta\int_{M_\eta}|u_{xx}u_x|.$$ Now we define $\widetilde{u}$ by
setting $\widetilde{u}=u$ on $M_\eta$ and $\widetilde{u}=0$
elsewhere. Then by Cauchy-Schwartz,
$$\Big|\intr u_{xx}u_x\sgn_\eta'u\Big|\leq \frac 1\eta\intr |\widetilde{u}_{xx}\widetilde{u}_x|\leq \frac 1\eta \|\widetilde{u}_{xx}\|_{L^2}\|\widetilde{u}_x\|_{L^2}.$$
The second estimate in (\ref{ele1}) and (\ref{reas2}) yield
$$\|\widetilde{u}_x\|_{L^2}\leq \|\widetilde{u}\|_{L^2}^{1/2}\|\widetilde{u}_{xx}\|_{L^2}^{1/2}=\Big(\int_{M_\eta}|u|^2\Big)^{1/4}
\Big(\int_{M_\eta}|u_{xx}|^2\Big)^{1/4}\leq c\eta^{1/2}\mes(M_\eta)^{1/2}$$ and
$$\|\widetilde{u}_{xx}\|_{L^2}\leq \|\widetilde{u}\|_{L^2}^{1/2}\|\widetilde{u}_{xxxx}\|_{L^2}^{1/2}=\Big(\int_{M_\eta}|u|^2\Big)^{1/4}
\Big(\int_{M_\eta}|u_{xxxx}|^2\Big)^{1/4}\leq c\eta^{1/2}\mes(M_\eta)^{1/2}.$$
Gathering these two last estimates we infer $$\Big|\intr
u_{xx}u_x\sgn_\eta'u\Big|\leq c\mes(M_\eta)\rightarrow 0$$ and so
$\intr u_{xxx}\sgn u=0$. Finally $$\partial_t\|u(t)\|_{L^1}\leq
0,$$ which complete the proof of Proposition \ref{decayl1}.
\end{proof}

\begin{corollary}\label{cor-l2} Let $u_0\in L^1(\R)\cap L^2(\R)$ and $u$ be a solution of (\ref{eq}) satisfying
(\ref{as2}). Then,
\begin{equation}\label{utl2}\forall t>0,\quad \|u(t)\|_{L^2}\leq c(1+t)^{-1/2\alpha}.\end{equation}
\end{corollary}
\begin{proof} If we multiply (\ref{eq}) by $u$ and then integrate the result over
$\R$,
$$\partial_t\|u(t)\|_{L^2}^2=-2\||D|^{\alpha/2}u\|_{L^2}^2\leq
0.$$ In particular, $\|u(t)\|_{L^2}\leq \|u_0\|_{L^2}$. For all
$t>0$, last equality allow us to write
\begin{align*}\partial_t\Big[t^{2/\alpha}\|u(t)\|_{L^2}^2\Big] &= \frac 2\alpha
t^{2/\alpha-1}\|u(t)\|_{L^2}^2+t^{2/\alpha}\partial_t\|u(t)\|_{L^2}^2\\ &= \frac 2\alpha
t^{2/\alpha-1}\|u(t)\|_{L^2}^2-2t^{2/\alpha}\intr|\xi|^\alpha|\hat{u}(t,\xi)|^2d\xi\\ &\leq \frac
2\alpha t^{2/\alpha-1}\intr|\hat{u} (t,\xi)|^2d\xi-2t^{2/\alpha}\int_{|\xi|>(\alpha t)^{-1/\alpha}}
|\xi|^\alpha|\hat{u}(t,\xi)|^2d\xi \\ &\leq \frac 2\alpha t^{2/\alpha-1}\intr|\hat{u}(t,\xi)|^2
d\xi-\frac 2\alpha t^{2/\alpha-1}\int_{|\xi|>(\alpha t)^{-1/\alpha}}|\hat{u}(t,\xi)|^2d\xi\\
&=\frac 2\alpha t^{2/\alpha-1}\int_{|\xi|<(\alpha t)^{-1/\alpha}}|\hat{u}(t,\xi)|^2d\xi\\
&\leq ct^{2/\alpha-1}\|u(t)\|_{L^1}^2\mes\{|\xi|<(\alpha
t)^{-1/\alpha}\}\\ &\leq ct^{1/\alpha-1}.
\end{align*}
The integration of this inequality over $[0,t]$ provides the
desired result.
\end{proof}

Now we show that if $\alpha\geq 1$, solutions of (\ref{eq})
satisfy the maximum principle. The restriction on $\alpha$ is
mainly due to the fact that one has $|D|^\alpha 1=0$ only if
$\alpha\geq 1$.

\begin{lemma}\label{lem-max} If $u$ is a solution to (\ref{eq}) with $\alpha\geq 1$ associated with
initial data $u_0\in L^\infty(\R)$, then
\begin{equation}\label{max}\inf u_0\leq u(t,x)\leq \sup u_0\end{equation} for a.e.
$(t,x)\in [0,\infty[\times\R$.
\end{lemma}

\begin{proof} Let $m=\inf u_0$, $M=\sup u_0$ and
$u^+=\max(0,u-M-\eps)$, $u^-=\min(0,u+m+\eps)$ for some $\eps>0$.
We multiply (\ref{eq}) by $u^+$ and integrate over $\R$ to get
\begin{equation}\label{multu+}\intr (u_t+u_{xxx}+|D|^\alpha u+uu_x)u^+=0.
\end{equation} On the support
of $u^+$, it is clear that $u_t=u_t^+$, $u_{x}=u_{x}^+$ and
$|D|^\alpha u=|D|^\alpha u^+$, this last equality follows from the
relation $|D|^\alpha 1=0$ for $\alpha\geq 1$. We deduce $\intr
u_tu^+=\frac 12\partial_t\|u^+(t)\|_{L^2}^2$, $\intr
u_{xxx}u^+=\intr u_{xxx}^+u^+=0$ and $\intr |D|^\alpha uu^+=\intr
|D|^\alpha u^+u^+=\||D|^{\alpha/2}u^+\|_{L^2}^2$ by Plancherel. On
the other hand, one has $\intr
uu_xu^+=\intr(u^++M+\eps)u_x^+u^+=0$. Inserting this into
(\ref{multu+}) and integrating over $[0,t]$ we get
$$\|u^+(t)\|_{L^2}^2+2\int_0^t\||D|^{\alpha/2}u^+(s)\|_{L^2}^2ds=\|u^+(0)\|_{L^2}^2=0$$
and thus $u^+(t)=0$ a.e.. Consequently, we have $u(t)\leq M+\eps$
for all $\eps>0$, and the second part of (\ref{max}) is proved.
The same arguments hold with $u^+$ replaced by $u^-$ and give the
first inequality.
\end{proof}

Following \cite{MR1643236}, we introduce for $\lambda>1$ the
following rescaled solution
$$u_\lambda(t,x)=\lambda u(\lambda^2t,\lambda x).$$ Obviously,
$u_\lambda$ satisfies the equation
$$\partial_tu_\lambda+\lambda^{-1}\partial_{xxx}u_\lambda+\lambda^{2-\alpha}|D|^\alpha
u_\lambda+u_\lambda\partial_x u_\lambda=0$$ with initial data $u_{0,\lambda}(x)=\lambda u_0(\lambda x)$.

\begin{lemma}\label{uniful}Let $u_0\in L^1(\R)\cap L^\infty(\R)$ and $u$ be a
solution of (\ref{eq}) with $1<\alpha<2$ satisfying (\ref{as2}).
For $j\geq 0$, $T>0$ and $0<t<T$, there exists $c=c(t,T)$ such
that for all $\lambda>1$, one has
$\|\partial_x^ju_\lambda(t)\|_{L^2}\leq c$.
\end{lemma}

\begin{proof} The method of proof is based on an induction on $j$. If $j=0$,
one easily deduce from Corollary \ref{cor-l2} and Lemma
\ref{lem-max} that $\|u(t)\|_{L^p}\leq ct^{-1/\alpha p}$ for
$2\leq p\leq \infty$ and thus
\begin{equation}\label{lpul}\|u_\lambda(t)\|_{L^p}\leq
c\lambda^{1-(1+2/\alpha)/p}t^{-1/\alpha p}.\end{equation} In
particular for $p=2$ and $\lambda>1$, $\|u_\lambda(t)\|_{L^2}\leq
c(t)$. Suppose now that the result is true for all $k<j$. Consider
$S_\alpha^\lambda(t)$ (resp. $G_\alpha^\lambda(t)$), the semigroup
generated by $\lambda^{-1}\partial_{xxx}+\lambda^{2-\alpha}
|D|^\alpha$ (resp. $\lambda^{2-\alpha} |D|^\alpha$) so that we
have for $0<t,t'<T$
\begin{equation}\label{intul}u_\lambda(t+t')=S_\alpha^\lambda(t)\ast
u_{\lambda}(t')-\frac
12\int_0^tS_\alpha^\lambda(t-s)\ast\partial_xu_\lambda^2(s+t')ds.\end{equation}
It is worth noticing that $\|S_\alpha^\lambda(t)\ast
f\|_{L^2}=\|G_\alpha^\lambda(t)\ast f\|_{L^2}$ and
\begin{equation}\label{gal}\|\partial_x^jG_\alpha^\lambda(t)\|_{L^p}=\|\partial_x^jG_\alpha(\lambda^{2-\alpha}t)\|_{L^p}
\end{equation}
for all $1\leq p\leq\infty$.
 Application of $\partial_x^j$ to
(\ref{intul}) and computing the $L^2$-norm lead to
\begin{multline}\label{hjul}\|\partial_x^ju_\lambda(t+t')\|_{L^2}\leq
\|\partial_x^jG_\alpha^\lambda(t)\ast
u_\lambda(t')\|_{L^2}\\+c\sum_{k=0}^j\int_0^t\|\partial_xG_\alpha^\lambda(t-s)\ast
\partial_x^ku_\lambda(s+t')
\partial_x^{j-k}u_\lambda(s+t')\|_{L^2}ds.\end{multline}
By the inductive hypothesis, the first term in the right-hand side
of (\ref{hjul}) is bounded by
$$\|\partial_xG_\alpha^\lambda(t)\ast\partial_x^{j-1}u_\lambda(t')\|_{L^2}\leq
ct^{-1/\alpha}\|\partial_x^{j-1}u_\lambda(t')\|_{L^2}\leq
c(t')t^{-1/\alpha}.$$ By symmetry, it is sufficient in the sum
$\sum_{k=0}^j$ in (\ref{hjul}) to consider the indexes
$k=0,...E(j/2)$. The case $k=0$ is a special case and has to be
treated separately. Using Young and H\"{o}lder inequalities and
next estimates (\ref{gal}) and (\ref{lpul}), we obtain
\begin{align}
\notag
&\|\partial_xG_\alpha^\lambda(t-s)\ast u_\lambda(s+t')\partial_x^ju_\lambda(s+t')\|_{L^2}\\ \notag &\quad\leq
\|\partial_xG_\alpha^\lambda(t-s)\|_{L^{2/(3-\alpha)}}
\|u_\lambda(s+t')\|_{L^{2/(\alpha-1)}}\|\partial_x^ju_\lambda(s+t')\|_{L^2}\\ &\notag\quad\leq [\lambda^{2-\alpha}(t-s)]^{-(\alpha+1)/2\alpha}
\lambda^{1-(\alpha-1)(1+2/\alpha)/2}(s+t')^{(1-\alpha)/2\alpha}\|\partial_x^ju_\lambda(s+t')\|_{L^2}\\ \label{intul1}
&\quad\leq c(s+t')(t-s)^{-(\alpha+1)/2\alpha}\|\partial_x^ju_\lambda(s+t')\|_{L^2}
\end{align}
since $-(2-\alpha)(\alpha+1)/2\alpha+1-(\alpha-1)(1+2/\alpha)/2=
0$. When $k\geq 1$, we use the inductive hypothesis combined with
(\ref{ele1}) to get
\begin{align}
\notag&\|\partial_xG_\alpha^\lambda(t-s)\ast
\partial_x^ku_\lambda(s+t')\partial_x^{j-k}u_\lambda(s+t')\|_{L^2}\\ \notag &\quad \leq
\|\partial_xG_\alpha^\lambda(t-s)\|_{L^1}\|\partial_x^ku_\lambda(s+t')\|_{L^\infty}\|\partial_x^{j-k}u_\lambda(s+t')\|_{L^2}\\
\label{intul2}&\quad\leq c(s+t')(t-s)^{-1/\alpha}.
\end{align}
Bounding $c(s+t')$ in (\ref{intul1})-(\ref{intul2}) by
$\sup_{0\leq s\leq T}c(s+t')$ and inserting these inequalities
into (\ref{hjul}) let us conclude that
\begin{multline*}\|\partial_x^ju_\lambda(t+t')\|_{L^2} \leq
c(t')t^{-1/\alpha}+c(t',T)\\+
c(t')\int_0^t(t-s)^{-(\alpha+1)/2\alpha}\|\partial_x^ju_\lambda(s+t')\|_{L^2}ds.
\end{multline*}
This implies by the generalized Gronwall lemma \cite{MR1274542}
that for $t'=t$, $$\|\partial_x^ju_\lambda(2t)\|_{L^2}\leq
c(t,T)$$ where $c(t,T)$ is independent of $\lambda>1$.
\end{proof}
As noticed in Section \ref{sec-main}, these uniform estimates (in
$\lambda$) imply uniform estimates in time of the solution.
\begin{corollary} Let $u_0\in L^1(\R)\cap L^\infty(\R)\cap
H^j(\R)$ for some $j\geq 0$. Assume that $u$ is a solution of
(\ref{eq}) with $1<\alpha<2$ satisfying (\ref{as2}). Then
assumption (\ref{as3}) is satisfied, i.e. $$\sup_{t\geq
0}\|\partial_x^ju(t)\|_{L^2}<\infty.$$
\end{corollary}
\begin{proof} First since $u_0\in H^j(\R)$, we have $u\in
C([0,\infty[ ; H^j(\R))$ and thus $\sup_{0\leq t\leq
1}\|\partial_x^ju(t)\|_{L^2}<\infty$. On the other hand, one
easily verifies that
$\|\partial_x^ju_\lambda(t)\|_{L^2}=\lambda^{j+1/2}\|\partial_x^ju(\lambda^2t)\|_{L^2}$.
Taking $t=1$ and $\lambda=t^{1/2}$ in this equality we deduce
$$t^{j/2+1/4}\|\partial_x^ju(t)\|_{L^2}=\|\partial_x^ju_\lambda(1)\|_{L^2}\leq
c$$ by Lemma \ref{uniful}. This implies for $t\geq 1$ that
$\|\partial_x^ju(t)\|_{L^2}\leq c$ as desired.
\end{proof}

\section{Decay of solutions to (\ref{eq})}\label{sec-decay}
In this section we prove Theorem \ref{th-decayrate} which has
already been shown in the special cases $(p,j)=(1,0)$ and
$(p,j)=(2,0)$ in the previous section.
\begin{lemma}\label{lemti}Let $u_0\in H^{j}(\R)\cap L^1(\R)$ and $u$ be a solution satisfying
(\ref{as2})-(\ref{as3}). Then, for all $t>1$ and $N\geq 1$,
\begin{multline}\label{estti}\Big\|\int_0^tS_\alpha(t-s)\ast\partial_xu^2(s)ds\Big\|_{\dot{H}^j}\\ \leq ct^{-1/2\alpha-j/\alpha}+
ct^{-\gamma}\sup_{t/2\leq s\leq
t}\|\partial_x^ju(s)\|_{L^2}^{1-1/N}+t^{-\gamma}\sup_{t/2\leq
s\leq t}\|\partial_x^ju(s)\|_{L^2}\end{multline} with
$\gamma=\gamma(\alpha)>0$.
\end{lemma}
\begin{corollary}\label{estuHj}If  $u_0\in H^{j}(\R)\cap L^1(\R)$ and if (\ref{as2})-(\ref{as3}) hold true,
$$\|u(t)\|_{\dot{H}^j}\leq c(1+t)^{-1/2\alpha-j/\alpha}$$ for any
$t>0$.

\end{corollary}
\begin{proof}[Proof of Lemma \ref{lemti}]One proceeds by induction on $j$. For $j=0$ we use the integral formulation
(\ref{duhamel}) and estimates (\ref{estSLp}) and (\ref{utl2}) :
\begin{multline*}\Big\|\int_0^tS_\alpha(t-s)\ast\partial_x u^2(s)ds\Big\|_{L^2} =2\|u(t)-S_\alpha(t)\ast u_0\|_{L^2}\\
\leq 2\|u(t)\|_{L^2}+2\|S_\alpha(t)\|_{L^2}\|u_0\|_{L^1}\leq ct^{-1/2\alpha}.\end{multline*}
Now assume the statement (and thus Corollary \ref{estuHj}) is true for the $k<j$. We split the left-hand side of (\ref{estti}) into
$$\Big\|\int_0^tS_\alpha(t-s)\ast\partial_xu^2(s)ds\Big\|_{\dot{H}^j}\leq \int_0^{t/2}\ldots ds+\int_{t/2}^t\ldots ds :=I+II.$$
By the Young inequality and estimates (\ref{estSLp}), (\ref{utl2}), we have
\begin{align*}I &\leq \int_0^{t/2}\|\partial_x^{j+1}S_\alpha(t-s)\|_{L^2}\|u(s)\|_{L^2}^2ds\\
&\leq c\int_0^{t/2}(t-s)^{-1/2\alpha-(j+1)/\alpha}(1+s)^{-1/\alpha}ds\\ &\leq ct^{-1/2\alpha-j/\alpha}\Big(t^{-1/\alpha}\int_0^t(1+s)^{-1/\alpha}ds\Big)
\end{align*}
and for $t>1$,$$t^{-1/\alpha}\int_0^t(1+s)^{-1/\alpha}ds\leq
c\left\{\begin{array}{lll}t^{-1/\alpha} &\rm{ if } &\alpha<1\\
t^{-1}\log t &\rm{ if } &\alpha=1\\ t^{1-2/\alpha} &\rm{ if }&
\alpha>1\end{array}\right.\leq c.$$ To estimate $II$, we use Plancherel and we split low and high frequencies,
\begin{align*}II &= c\int_{t/2}^t\Big(\intr e^{-2(t-s)|\xi|^\alpha}|\xi|^{2(j+1)}|\widehat{u^2}(s,\xi)|^2d\xi\Big)^{1/2}ds\\
&\leq c\int_{t/2}^t\Big(\int_{|\xi|<1}\ldots d\xi\Big)^{1/2}ds+c\int_{t/2}^t\Big(\int_{|\xi|>1}\ldots d\xi\Big)^{1/2}ds:=II_1+II_2.\end{align*}
If $|\xi|<1$, then $e^{-2|\xi|^\alpha}\geq e^{-2}$, hence
\begin{align*}II_1 &\leq c\int_{t/2}^t\Big(\intr e^{-2(1+t-s)|\xi|^\alpha}|\xi|^{2(j+1)}|\widehat{u^2}(s,\xi)|^2d\xi\Big)^{1/2}
\\ &= c\int_{t/2}^t\|\partial_x S_\alpha(1+t-s)\ast\partial_x^ju^2(s)\|_{L^2}ds\\ &\leq c\int_{t/2}^t\|\partial_xS_\alpha(1+t-s)\|_{L^2}\|\partial_x^ju^2(s)\|_{L^1}ds\\
&\leq c\int_{t/2}^t(1+t-s)^{-3/2\alpha}\sum_{k=0}^j\|\partial_x^ku(s)\|_{L^2}\|\partial_x^{j-k}u(s)\|_{L^2}ds.
\end{align*}
Corollary \ref{estuHj} with $k<j$ implies that
\begin{align}\notag \sum_{k=0}^j
\|\partial_x^ku(s)\|_{L^2}\|\partial_x^{j-k}u(s)\|_{L^2} &\leq
c\sum_{k=1}^{j-1}(1+s)^{-1/2\alpha-k/\alpha}(1+s)^{-1/2\alpha-(j-k)/\alpha}\\\notag
&\quad+c\|u(s)\|_{L^2}\|\partial_x^ju(s)\|_{L^2}\\ \label{eq2t}
&\leq
c(1+s)^{-1/\alpha-j/\alpha}+(1+s)^{-1/2\alpha}\|\partial_x^ju(s)\|_{L^2}.\end{align}
For the contribution of the first term in (\ref{eq2t}), we have
$$\int_{t/2}^t(1+t-s)^{-3/2\alpha}(1+s)^{-1/\alpha-j/\alpha}ds\leq ct^{-1/2\alpha-j/\alpha}\Big(t^{-1/2\alpha}\int_0^t(1+s)^{-3/2\alpha}ds\Big)$$ and
for $t>1$,
\begin{equation}\label{eps}t^{-1/2\alpha}\int_0^t(1+s)^{-3/2\alpha}ds\leq
c\left\{\begin{array}{lll}t^{-1/2\alpha}&\rm{if}& \alpha<3/2\\
t^{-1/3}\log t &\rm{if} & \alpha=3/2\\ t^{1-2/\alpha}&\rm{if} &
\alpha>3/2\end{array}\right. \leq c.\end{equation} For the second
one, one can write
\begin{multline*}\int_{t/2}^t(1+t-s)^{-3/2\alpha}(1+s)^{-1/2\alpha}\|\partial_x^ju(s)\|_{L^2}ds
\\\leq c\Big(t^{-1/2\alpha}\int_0^t(1+s)^{-3/2\alpha}ds\Big)\sup_{t/2\leq s\leq t}\|\partial_x^ju(s)\|_{L^2} \leq ct^{-\gamma}\sup_{t/2\leq s\leq t}\|u(s)\|_{\dot{H}^j}
\end{multline*}
 in view of (\ref{eps}). Term  $II_2$ is bounded by
\begin{align*}II_2 &\leq \int_{t/2}^t\Big(\intr e^{-2(t-s)}|\xi|^{2(j+1)}|\widehat{u^2}(s,\xi)|^2d\xi\Big)^{1/2}ds\\
&=c\int_{t/2}^t e^{-(t-s)}\|\partial_x^{j+1}u^2(s)\|_{L^2}ds\\ &\leq c \int_{t/2}^te^{-(t-s)}\sum_{k=0}^{j+1}\|\partial_x^ku(s)\partial_x^{j+1-k}u(s)\|_{L^2}ds.
\end{align*}
By symmetry, it suffices in the previous sum to consider the
values $k=0,1,\ldots,E((j+1)/2)$.  When $k=0$, assumption
(\ref{as3}) and Lemma \ref{elelem} provide
\begin{align*}\|u(s)\partial_x^{j+1}u(s)\|_{L^2} &\leq\|u(s)\|_{L^\infty}\|\partial_x^{j+1}u(s)\|_{L^2}\\
&\leq c\|u(s)\|_{L^2}^{1/2}\|u_x(s)\|_{L^2}^{1/2} \|\partial_x^ju(s)\|_{L^2}^{1-1/N}\|\partial_x^{j+N}u(s)\|_{L^2}^{1/N}\\
&\leq c(1+s)^{-1/4\alpha}\|\partial_x^ju(s)\|_{L^2}^{1-1/N}
\end{align*}
for any $N\geq 1$. For $k=1$, we have by similar calculations
\begin{align*}\|u_x(s)\partial_x^ju(s)\|_{L^2} &\leq \|u_x(s)\|_{L^\infty}\|\partial_x^ju(s)\|_{L^2}\\
&\leq c\|u(s)\|_{L^2}^{1/4}\|u_{xx}(s)\|_{L^2}^{3/4}\|\partial_x^ju(s)\|_{L^2}\\
&\leq c(1+s)^{-1/8\alpha}\|\partial_x^ju(s)\|_{L^2}.
\end{align*}
Note that if $k=2$, we must have $j\geq 3$. If $j\geq 4$, one has
by the inductive hypothesis
\begin{align*}\|\partial_x^2u(s)\partial_x^{j-1}u(s)\|_{L^2}
&\leq c\|\partial_x^2u(s)\|_{L^\infty}\|\partial_x^{j-1}u(s)\|_{L^2}\\& \leq c\|\partial_x^2u(s)\|_{L^2}^{1/2}\|\partial_x^3u(s)\|_{L^2}^{1/2}(1+s)^{-1/2\alpha-(j-1)/\alpha}
\\ &\leq c(1+s)^{-1/2\alpha-j/\alpha}.
\end{align*}
If $j=3$, then
\begin{align*}\|u_{xx}(s)u_{xx}(s)\|_{L^2} &\leq
\|u_{xx}(s)\|_{L^\infty}\|u_{xx}(s)\|_{L^2}\\ &\leq c\|u_{xx}(s)\|_{L^2}^{1/2}\|u_{xxx}(s)\|_{L^2}^{1/2}\|u_x(s)\|_{L^2}^{1/2}\|u_{xxx}(s)\|_{L^2}^{1/2}\\
&\leq c(1+s)^{-2/\alpha}\|\partial_x^ju(s)\|_{L^2}.
\end{align*}
 In the end
for $k\geq 3$ (and thus $j\geq 5$),
\begin{align*}\|\partial_x^ku(s)\partial_x^{j+1-k}u(s)\|_{L^2} &\leq \|\partial_x^ku(s)\|_{L^2}\|\partial_x^{j+1-k}u(s)\|_{L^\infty}\\
&\leq \|\partial_x^ku(s)\|_{L^2}\|\partial_x^{j+1-k}u(s)\|_{L^2}^{1/2}\|\partial_x^{j+2-k}u(s)\|_{L^2}^{1/2}\\
&\leq c(1+s)^{-1/2\alpha-k/\alpha+(-1/2\alpha-(j+1-k)/\alpha)/2+(-1/2\alpha-(j+2-k)/\alpha)/2}\\
&\leq c(1+s)^{-5/2\alpha-j/\alpha}\leq c(1+s)^{-1/2\alpha-j/\alpha}
\end{align*}
This allows us to conclude that \begin{align*}II_2 &\leq
c\int_{t/2}^te^{-(t-s)}[(1+s)^{-1/2\alpha-j/\alpha}+s^{-\gamma}\|\partial_x^ju(s)\|_{L^2}+(1+s)^{-\gamma}\|\partial_x^ju(s)\|_{L^2}^{1-1/N}]ds\\
&\leq
c[t^{-1/2\alpha-j/\alpha}+t^{-\gamma}\sup_{t/2\leq s\leq t}\|\partial_x^ju(s)\|_{L^2}+t^{-\gamma}\sup_{t/2\leq s\leq t}\|\partial_x^ju(s)\|_{L^2}^{1-1/N}]\int_0^te^{-(t-s)}ds\\
&\leq ct^{-1/2\alpha-j/\alpha}+t^{-\gamma}\sup_{t/2\leq s\leq
t}\|\partial_x^ju(s)\|_{L^2}+ct^{-\gamma}\sup_{t/2\leq s\leq
t}\|\partial_x^ju(s)\|_{L^2}^{1-1/N}.
\end{align*}
\end{proof}

In order to prove Corollary \ref{estuHj}, we need the following elementary result.
\begin{lemma}\label{lemel}Let $f:\R^+\rightarrow\R^+$ bounded, and $0<\gamma<\beta$ and $N\geq 1$.
We assume $$\forall t\geq 1,\quad f(t)\leq
ct^{-\beta}+ct^{-\gamma}\sup_{s\sim t}f(s)^{1-1/N}.$$ Then for $t$
and $N$ large enough, $f(t)\leq ct^{-\beta}$.
\end{lemma}
\begin{proof}We show by induction that for all $n\geq 0$, $f(t)\leq ct^{-\min(\beta,\gamma(1-N)(1-\frac 1N)^n+\gamma N)}$.
Thus for $n$ large enough, one obtains $f(t)\leq ct^{-\min(\beta,
\gamma N+1)}$ and it suffices to choose $N$ so that $\beta\leq
\gamma N+1$.
\end{proof}
\begin{proof}[Proof of Corollary \ref{estuHj}] By (\ref{as3}), we only need to consider $t$ large enough. Using (\ref{estSLp}) and Lemma \ref{lemti}, it follows that
\begin{align*}
\|\partial_x^ju(t)\|_{L^2} &\leq \|\partial_x^jS_\alpha(t)\ast u_0\|_{L^2}+\Big\|\frac 12\int_0^t\partial_x^jS_\alpha(t-s)\ast
\partial_xu^2(s)ds\Big\|_{L^2}\\ &\leq ct^{-1/2\alpha-j/\alpha}+ct^{-\gamma}\sup_{t/2\leq s\leq t}\|\partial_x^ju(s)\|_{L^2}+ct^{-\gamma}
\sup_{t/2\leq s\leq t}\|\partial_x^ju(s)\|_{L^2}^{1-1/N}.
\end{align*}
Letting $t\rightarrow \infty$, we deduce
$\|\partial_x^ju(t)\|_{L^2}\rightarrow 0$. For $t\gg 1$, we thus
have $\|\partial_x^ju(t)\|_{L^2}\leq 1$ and
$$\|\partial_x^ju(t)\|_{L^2}\leq
ct^{-1/2\alpha-j/\alpha}+ct^{-\gamma}\sup_{t/2\leq s\leq
t}\|\partial_x^ju(s)\|_{L^2}^{1-1/N}.$$ Applying Lemma \ref{lemel}
with $f(t)=\|\partial_x^ju(t)\|_{L^2}$ and
$\beta=1/2\alpha+j/\alpha$, we obtain the desired result.
\end{proof}

\vskip 0.5cm

\begin{proof}[Proof of Theorem \ref{th-decayrate}] The result is already proved in the case $p=2$. When $p=\infty$,
we use (\ref{ele1}) and Corollary \ref{estuHj} to get
$$\|u(t)\|_{\dot{H}^{\infty,j}}\leq c\|u(t)\|_{\dot{H}^j}^{1/2}\|u(t)\|_{\dot{H}^{j+1}}^{1/2}\leq c(1+t)^{-1/\alpha-j/\alpha}.$$
The other cases follow by an interpolation argument.
\end{proof}

\section{Asymptotic expansion}\label{sec-higher}
\subsection{First order}

In this subsection we prove Theorem \ref{th-ordre1}. As previously, it suffices to show the result when $p=2$ and $u_0\in H^{j+2}(\R)\cap L^1(\R)$.

First, since $u\in C_b(\R^+,H^j(\R))$,
$$\|u(t)-S_\alpha(t)\ast u_0\|_{\dot{H}^j}\leq \|u(t)\|_{\dot{H}^j}+\|G_\alpha(t)\|_{L^1}\|u_0\|_{\dot{H}^j}\leq c$$
and we reduce to consider the case $t\geq 1$. Using the integral formulation of (\ref{eq}), we have
\begin{align*}\|u(t)-S_\alpha(t)\ast u_0\|_{\dot{H}^j} &\leq \frac 12\int_0^t\|\partial_x^j S_\alpha(t-s)\ast\partial_x u^2\|_{L^2}ds \\
 &= \int_0^{t/2}\ldots ds+\int_{t/2}^t\ldots ds :=I+II.
\end{align*}
Term $I$ is bounded by
\begin{align*}I &\leq c\int_0^{t/2}\|\partial_x^{j+1}S_\alpha(t-s)\|_{L^2}\|u(s)\|_{L^2}^2ds\\
&\leq c\int_0^{t/2}(t-s)^{-1/2\alpha-(j+1)/\alpha}(1+s)^{-1/\alpha}ds\\ &\leq ct^{-1/2\alpha-(j+1)/\alpha}\int_0^t(1+s)^{-1/\alpha}ds\\
&\leq c \left\{\begin{array}{lll}t^{(-1/2\alpha-j/\alpha)-1/\alpha} &\rm{if} & \alpha<1,\\ t^{(-1/2-j)-1}\log(t)
 &\rm{if} & \alpha=1,\\ t^{(-1/2\alpha-j/\alpha)-(2/\alpha-1)} &\rm{if} & \alpha>1.\end{array}\right.
\end{align*}
To estimate $II$ we use Plancherel and we split low and high frequencies,
\begin{align*}II &= c\int_{t/2}^t\Big(\intr e^{-2(t-s)|\xi|^\alpha}|\xi|^{2(j+1)}|\widehat{u^2}(s,\xi)|^2d\xi\Big)^{1/2}ds\\
 &\leq c\int_{t/2}^t\Big(\int_{|\xi|<1}\ldots d\xi\Big)^{1/2}ds+c\int_{t/2}^t\Big(\int_{|\xi|>1}\ldots d\xi\Big)^{1/2}ds:=II_1+II_2.\end{align*}
$II_1$ is treated as follows
\begin{align*}II_1 &\leq c\int_{t/2}^t\|S_\alpha(1+t-s)\|_{L^2}\|\partial_x^{j+1}u^2(s)\|_{L^1}ds\\
&\leq c\int_{t/2}^t(1+t-s)^{-1/2\alpha} (1+s)^{-2/\alpha-j/\alpha}ds\\ &\leq ct^{-2/\alpha-j/\alpha}\int_0^t(1+s)^{-1/2\alpha}ds\\
&\leq c\left\{\begin{array}{lll}t^{-2/\alpha-j/\alpha} &\rm{if} & \alpha<1/2,\\ t^{-4-2j}\log t
&\rm{if} & \alpha=1/2,\\ t^{(-1/2\alpha-j/\alpha)+1-2/\alpha} &\rm{if} & \alpha>1/2,\end{array}\right.
\\ &\leq c\left\{\begin{array}{lll}t^{(-1/2\alpha-j/\alpha)-1/\alpha} &\rm{if} & \alpha<1,\\ t^{(-1/2-j)-1}\log(t)
&\rm{if} & \alpha=1,\\ t^{(-1/2\alpha-j/\alpha)-(2/\alpha-1)} &\rm{if} & \alpha>1.\end{array}\right.
\end{align*}
For the last term, we have
\begin{align*}II_2 &\leq c\int_{t/2}^te^{-(t-s)}\|\partial_x^{j+1}u^2(s)\|_{L^2}ds\\
&\leq c\int_{t/2}^t e^{-(t-s)}(1+s)^{-1/2\alpha-j/\alpha-2/\alpha}\\ &\leq ct^{(-1/2\alpha-j/\alpha)-2/\alpha},
\end{align*}
which is acceptable.

\subsection{Higher orders}
Here we find higher orders terms in the asymptotic expansion of
the solution to (\ref{eq}), i.e. we give a demonstration of
Theorems \ref{th-ordre2.1} and \ref{th-ordre2.2}.

\subsubsection{The case $0<\alpha<1$} First consider the case $0<\alpha<1$, our proof
follows Karch's one \cite{MR1727212} (see also \cite{MR1708995}).
\begin{proof}[Proof of Theorem \ref{th-ordre2.1} (i)]
By interpolation, we only need to consider the case $p=2$ and $u_0\in H^{j+2}(\R)$. Split the quantity
\begin{align*}&\Big\|u(t)-S_\alpha(t)\ast u_0+\frac 12\Big(\int_0^\infty\intr u^2(s,y)dyds\Big)\partial_xG_\alpha(t)\Big\|_{\dot{H}^j}\\
&\leq \frac 12\Big\|\int_0^t\partial_x[S_\alpha(t-s)-G_\alpha(t-s)]\ast u^2(s)ds\Big\|_{\dot{H}^j}\\
&\quad+\frac 12\Big\|\int_0^t\partial_xG_\alpha(t-s)\ast u^2(s)ds-\Big(\int_0^\infty\intr u^2(s,y)dyds\Big)\partial_xG_\alpha(t)\Big\|_{\dot{H}^j}\\ &:=I+II.
\end{align*}
To estimate $I$, we write
\begin{align*}I &\leq c\int_0^t\|\partial_x^{j+1}[S_\alpha(t-s)-G_\alpha(t-s)]\ast u^2(s)\|_{L^2}ds\\ & = \int_0^{t/2}\ldots ds+\int_{t/2}^t\ldots ds :=I_1+I_2.
\end{align*}
Concerning $I_1$, we use (\ref{estSG}) with $N=0$,
\begin{align*}I_1 &\leq c\int_0^{t/2}\|\partial_x^{j+1}[S_\alpha(t-s)-G_\alpha(t-s)]\|_{L^2}\|u(s)\|_{L^2}^2ds\\
&\leq c\int_0^{t/2}(t-s)^{-1/2\alpha-(j+1)/\alpha+1-3/\alpha}(1+s)^{-1/\alpha}ds\\ &\leq ct^{-1/2\alpha-j/\alpha-1/\alpha}t^{1-3/\alpha},
\end{align*}
which shows that $t^{1/2\alpha+j/\alpha+1/\alpha}I_1\rightarrow 0$.
To deal with the integrand over $[t/2,t]$, we note that $\|[S_\alpha(t-s)-G_\alpha(t-s)]\ast u^2(s)\|_{\dot{H}^{j+1}}\leq c\|u^2(s)\|_{\dot{H}^{j+1}}$, hence
\begin{align*}I_2 &\leq c\int_{t/2}^t\|\partial_x^{j+1}u^2(s)\|_{L^2}ds\\ &\leq c\int_{t/2}^t(1+s)^{-1/2\alpha-j/\alpha-2/\alpha}ds\\
&\leq ct^{-1/2\alpha-j/\alpha-1/\alpha}t^{1-1/\alpha},
\end{align*}
which is acceptable.
Now we estimate term $II$ by
\begin{align*}II &\leq \frac 12\Big\|\Big(\int_{t}^\infty\intr u^2(s,y)dyds\Big)\partial_xG_\alpha(t)\Big\|_{\dot{H}^j}\\
&\quad+ \frac 12\Big\|\int_0^{t}\Big[\partial_xG_\alpha(t-s)\ast u^2(s)-\Big(\intr u^2(s,y)dy\Big)\partial_xG_\alpha(t)\Big]ds\Big\|_{\dot{H}^j}\\ &:= II_1+II_2.
\end{align*}
Obviously,
$$ II_1 \leq c\int_{t}^\infty\|u(s)\|_{L^2}^2ds\|\partial_x^{j+1}G_\alpha(t)\|_{L^2}\leq ct^{(-1/2\alpha-j/\alpha)-1/\alpha} \int_{t}^\infty(1+s)^{-1/\alpha}ds
$$
and it is clear that $\int_{t}^\infty(1+s)^{-1/\alpha}ds\rightarrow 0$ as $t\rightarrow\infty$. To estimate $II_2$
one fixes $\delta>0$ and we bound it by
\begin{align*}II_2 &\leq c\Big\|\int_0^{t}\Big(\intr\partial_x[G_\alpha(t-s,\cdot-y)-G_\alpha(t,\cdot)]u^2(s,y)dy\Big)ds\Big\|_{\dot{H}^j}\\
&\leq c\int_0^t\Big\|\intr\partial_x^{j+1}[G_\alpha(t-s,\cdot-y)-G_\alpha(t,\cdot)]u^2(s,y)dy\Big\|_{L^2}ds\\
&= \int_0^{\delta t}\ldots ds+\int_{\delta t}^t\ldots ds\\ &= II_{21}+II_{22}.
\end{align*}
Then we split $II_{21}$ in two parts,
\begin{align*}II_{21} &\leq c\int_{[0,\delta t]\times\R} \|\partial_x^{j+1}[G_\alpha(t-s,\cdot-y)-G_\alpha(t,\cdot)]u^2(s,y)\|_{L^2}dsdy\\
&= c\int_{\Omega_1}\ldots dsdy+c\int_{\Omega_2}\ldots dyds\\ &=II_{211}+II_{212},
\end{align*}
where
\begin{align*}\Omega_1 &= [0,\delta t]\times [-\delta t^{1/\alpha},+\delta t^{1/\alpha}],\\
\Omega_2 &= [0,\delta t]\times(]-\infty,-\delta t^{1/\alpha}[\cup ]+\delta t^{1/\alpha},\infty[).
\end{align*}
For all $(s,y)\in \Omega_1$, a straightforward calculation provides
\begin{multline*}\|\partial_x^{j+1}[G_\alpha(t-s,\cdot-y)-G_\alpha(t,\cdot)]\|_{L^2}
\\=t^{-1/2\alpha-j/\alpha-1/\alpha}\|\partial_x^{j+1}[G_\alpha(1-s/t,\cdot-yt^{-1/\alpha})-G_\alpha(1,\cdot)]\|_{L^2}.
\end{multline*}
Hence, using the continuity of the translation on $L^2$, for all $\eps>0$, we can find a $\delta>0$ such that
\begin{multline*} t^{(1/2\alpha+j/\alpha)+1/\alpha}\sup_{(s,y)\in\Omega_1}\|\partial_x^{j+1}[G_\alpha(t-s,\cdot-y)-G_\alpha(t,\cdot)]\|_{L^2}\\
 \leq \sup_{\substack{0\leq \tau\leq \delta\\ |z|\leq \delta}}\|\partial_x^{j+1}[G_\alpha(1-\tau,\cdot-z)-G_\alpha(1,\cdot)]\|_{L^2}\leq \eps.
\end{multline*}
We deduce
$$t^{(1/2\alpha+j/\alpha)+1/\alpha}II_{211}\leq c\eps\int_0^{\delta t}\|u(s)\|_{L^2}^2ds\leq c\eps\int_0^{\delta t}(1+s)^{-1/\alpha}ds\leq c\eps.$$
Now for any $(s,y)\in\Omega_2$, we have
\begin{align*}\|\partial_x^{j+1}[G_\alpha(t-s,\cdot-y)-G_\alpha(t,\cdot)]\|_{L^2}
&\leq \|\partial_x^{j+1}G_\alpha(t-s)\|_{L^2}+\|\partial_x^{j+1}G_\alpha(t)\|_{L^2}\\ &\leq ct^{-1/2\alpha-(j+1)/\alpha},
\end{align*} which yields
$$t^{(1/2\alpha+j/\alpha)+1/\alpha}II_{212}\leq c\int_0^\infty\int_{|y|\geq \delta t^{1/\alpha}}u^2(s,y)dyds\rightarrow 0$$
by the dominated convergence theorem.\\ It remains to estimate $II_{22}$, we have
\begin{align*}II_{22} &= c\int_{\delta t}^t\|\partial_x^{j+1}G_\alpha(t-s)\ast u^2(s)-\|u(s)\|_{L^2}^2\partial_x^{j+1}G_\alpha(t)\|_{L^2}ds\\
&\leq c\int_{\delta t}^t\|\partial_x^{j+1}G_\alpha(t-s)\ast u^2(s)\|_{L^2}ds+c\int_{\delta t}^t(1+s)^{-1/\alpha}ds\|\partial_x^{j+1}G_\alpha(t)\|_{L^2}\\
&= II_{221}+II_{222}.
\end{align*}
The first term is bounded by
\begin{align*}II_{221} &\leq c\int_{\delta t}^t\Big(\intr |\xi|^{2(j+1)}e^{-2(t-s)|\xi|^\alpha}|\widehat{u^2}(s,\xi)|^2d\xi\Big)^{1/2}ds\\
&\leq c\int_{\delta t}^t[\|G_\alpha(1+t-s)\|_{L^2}\|\partial_x^{j+1}u^2(s)\|_{L^1}+e^{-(t-s)}\|\partial_x^{j+1}u^2(s)\|_{L^2}]ds\\
&\leq c\int_{\delta t}^t[(1+t-s)^{-1/2\alpha}(1+s)^{-2/\alpha-j/\alpha}+e^{-(t-s)}(1+s)^{-5/2\alpha-j/\alpha}]ds\\
&\leq ct^{-1/2\alpha-j/\alpha-1/\alpha}\Big(t^{-1/2\alpha}\int_0^t(1+s)^{-1/2\alpha}ds\Big)+ct^{-5/2\alpha-j/\alpha}
\end{align*}
and thus $t^{1/2\alpha+j/\alpha+1/\alpha}II_{221}\rightarrow 0$. On the other hand, we have immediately
$$II_{222}\leq ct^{-1/2\alpha-j/\alpha-1/\alpha}t^{1-1/\alpha},$$ which achieves the proof of (\ref{asympa1}).
\end{proof}

\subsubsection{The case $\alpha=1$} The proof of (\ref{asympa2}) uses the same
arguments together with the following result.
\begin{lemma}\label{limu1} Under the assumptions of Theorem \ref{th-ordre2.1} (ii),
$$\lim_{t\rightarrow \infty}\frac{1}{\log t}\int_0^t\intr u^2(s,y)dyds=\frac{M^2}{2\pi}.$$
\end{lemma}
\begin{proof}
First note that $$\frac{1}{\log t}\int_0^1\intr u^2(s,y)dyds\leq \frac{c}{\log t}\int_0^1(1+s)^{-1}ds\leq
\frac{c}{\log t}\rightarrow 0$$ and it remains to calculate the limit as $t\rightarrow\infty$ of
\begin{align}\label{limeval}\notag\frac{1}{\log t}\int_1^t\intr u^2(s,y)dyds &= \frac{1}{\log t}\int_1^t\intr (u^2(s,y)-(MG_1(s,y))^2)dyds\\
&\quad+\frac{1}{\log t}\int_1^t\intr(MG_1(s,y))^2dyds.
\end{align}
Using Theorem \ref{th-ordre1} as well as estimate (\ref{estSGlin}), we get for all $s> 1$
\begin{align*}\intr|u^2(s,y)-(MG_1(s,y))^2|dy &\leq \|u(s)+MG_1(s)\|_{L^2}\|u(s)-MG_1(s)\|_{L^2}\\ &\leq cs^{-1/2}\big(\|u(s)-S_1(s)\ast u_0\|_{L^2}\\
&\quad +\|S_1(s)\ast u_0-MG_1(s)\|_{L^2}\big)\\ &\leq cs^{-1/2}(s^{-3/2}\log s+s^{-3/2})\\ &\leq cs^{-2}\log s.
\end{align*}
It follows that
$$\frac{1}{\log t}\int_1^t\intr |u^2(s,y)-(MG_1(s,y))^2|dyds\leq \frac{c}{\log t}\int_1^ts^{-2}\log sds\rightarrow 0$$ by dominated convergence.
The last term in (\ref{limeval}) is equal to
\begin{align*}\frac{1}{\log t}\int_1^t\intr (MG_1(s,y))^2dyds &= \frac{M^2}{\log t}\int_1^t\intr s^{-2}(G_1(1,y/s))^2dyds\\
&= \frac{M^2}{\log t}\int_1^t\frac{ds}{s}\intr (G_1(1,x))^2dx\\ &= M^2\|G_1(1)\|_{L^2}^2\\ &=\frac{M^2}{2\pi}.
\end{align*}
\end{proof}

\begin{proof}[Proof of Theorem \ref{th-ordre2.1} (ii)] It is sufficient to show that
$$\frac{t^{3/2+j}}{\log t}\Big\|\int_0^t\partial_xS_1(t-s)\ast u^2(s)ds-\frac{M^2}{2\pi}(\log t)\partial_xG_1(t)\Big\|_{\dot{H}^j}\rightarrow 0.$$
for all $j\geq 0$.
As in Theorem \ref{th-ordre2.1} (i), we can replace $S_\alpha(t-s)$ by $G_\alpha(t-s)$ by writing
\begin{align*}&\Big\|\int_0^t\partial_xS_1(t-s)\ast u^2(s)ds-\frac{M^2}{2\pi}(\log t)\partial_xG_1(t)\Big\|_{\dot{H}^j}\\
&\leq \Big\|\int_0^t\partial_x[S_1(t-s)-G_1(-t-s)]\ast u^2(s)ds\Big\|_{\dot{H}^j}\\&\quad +
\Big\|\int_0^t\partial_xG_1(t-s)\ast u^2(s)ds-\frac{M^2}{2\pi}(\log t)\partial_xG_1(t)\Big\|_{\dot{H}^j}
\end{align*}
and using (\ref{estSG}). Last term in the previous inequality is bounded by
\begin{multline*}\leq \Big\|\int_0^t\partial_xG_1(t-s)\ast u^2(s)ds-\Big(\int_0^t\intr u^2(s,y)dyds\Big)\partial_xG_1(t)\Big\|_{\dot{H}^j}\\
+\Big\|\Big(\int_0^t\intr u^2(s,y)dyds\Big)\partial_xG_1(t)-\frac{M^2}{2\pi}(\log t)\partial_xG_1(t)\Big\|_{\dot{H}^j}.\end{multline*}
The first term is estimated exactly in the same way that $II_2$ in Theorem \ref{th-ordre2.1} (i) and for the second one, Lemma \ref{limu1} provides
\begin{align*}&\frac{t^{3/2+j}}{\log t}\Big\|\Big(\int_0^t\intr u^2(s,y)dyds\Big)\partial_xG_1(t)-\frac{M^2}{2\pi}
(\log t)\partial_xG_1(t)\Big\|_{\dot{H}^j}\\ &\leq t^{3/2+j}\Big|\frac{1}{\log t}\int_0^t\intr u^2(s,y)dyds-\frac{M^2}{2\pi}\Big|\|\partial_x^{j+1}G_1(t)\|_{L^2}\\
&\leq c\Big|\frac{1}{\log t}\int_0^t\intr u^2(s,y)dyds-\frac{M^2}{2\pi}\Big|\rightarrow 0.
\end{align*}
\end{proof}

\subsubsection{The case $1<\alpha<2$}
Finally we consider the case $1<\alpha<2$.
\begin{proof}[Proof of Theorem \ref{th-ordre2.2}] We prove the result when $p=2$ and $u_0\in H^{j+2}(\R)$.

\noindent {\bf Step 1.} $\|F^n(t)\|_{\dot{H}^j}$ decays like $\|u(t)\|_{\dot{H}^j}$.\\
If $n=0$, then for all $j\geq 0$,
$\|F^0(t)\|_{\dot{H}^j}=\|\partial_x^jS_\alpha(t)\ast
u_0\|_{L^2}\leq c(1+t)^{-1/2\alpha-j/\alpha}.$ Let $n\geq 0$ such
that for all $j\geq 0$, $\|F^n(t)\|_{\dot{H}^j}\leq
c(1+t)^{-1/2\alpha-j/\alpha}$ . Then, for any $t\leq 1$,
    \begin{align*}\|F^{n+1}(t)\|_{\dot{H}^j} &\leq \|S_\alpha(t)\ast u_0\|_{\dot{H}^j}+\int_0^t\|S_\alpha(t-s)\ast
    \partial_x(F^n(s))^2\|_{\dot{H}^j}ds\\ &\leq \|G_\alpha(t)\|_{L^1}\|u_0\|_{\dot{H}^j}+\int_0^1\|G_\alpha(t-s)\|_{L^1}\|\partial_x^{j+1}(F^n(s))\|_{L^2}ds\\ &\leq c.
    \end{align*}
    Now assume $t>1$. We have
    \begin{align*}\|F^{n+1}(t)\|_{\dot{H}^j} &\leq \|S_\alpha(t)\ast u_0\|_{\dot{H}^j}+\int_0^t\|S_\alpha(t-s)\ast
     \partial_x(F^n(s))^2\|_{\dot{H}^j}ds\\ &\leq c(1+t)^{-1/2\alpha-j/\alpha}+\int_0^{t/2}\ldots ds+\int_{t/2}^t\ldots ds.
    \end{align*}
    The integrand over $[0,t/2]$ is estimated as follows
    \begin{align*}\int_0^{t/2}\ldots ds &\leq \int_{t/2}^t\|\partial_x^{j+1}S_\alpha(t-s)\|_{L^2}\|F^n(s)\|_{L^2}^2ds\\
    &\leq c\int_0^{t/2}(t-s)^{-1/2\alpha-(j+1)/\alpha}(1+s)^{-1/\alpha}ds\\ &\leq ct^{-1/2\alpha-j/\alpha}\Big(t^{-1/\alpha}\int_0^t(1+s)^{-1/\alpha}ds\Big)\\
    &\leq ct^{-1/2\alpha-j/\alpha}.
    \end{align*}
    For the second one, one splits
    \begin{align*}\int_{t/2}^t\ldots ds &= c\int_{t/2}^t\Big(\intr |\xi|^{2(j+1)}e^{-2(t-s)|\xi|^\alpha}|\widehat{(F^n(s))^2}(\xi)|^2d\xi\Big)^{1/2}ds\\
    &\leq c\int_{t/2}^t\Big(\int_{|\xi|<1}\ldots d\xi\Big)^{1/2}ds+\int_{t/2}^t\Big(\int_{|\xi|>1}\ldots d\xi\Big)^{1/2}ds\\ &:= I+II.
    \end{align*}
    Term $I$ is bounded by
    \begin{align*}I &\leq c\int_{t/2}^t\|\partial_x^{j+1}S_\alpha(1+t-s)\ast F^n(s)\|_{L^2}ds\\
    &\leq c\int_{t/2}^t\|S_\alpha(1+t-s)\|_{L^2}\|\partial_x^{j+1}(F^n(s)^2\|_{L^1}ds\\
    &\leq c\int_{t/2}^t(1+t-s)^{-1/2\alpha}\sum_{k=0}^{j+1}\|\partial_x^kF^n(s)\|_{L^2}\|\partial_x^{j+1-k}F^n(s)\|_{L^2}ds\\
    &\leq c\int_{t/2}^t(1+t-s)^{-1/2\alpha}(1+s)^{-2/\alpha-j/\alpha}ds\\ &\leq ct^{-1/2\alpha-j/\alpha}\Big(t^{-3/2\alpha}\int_0^t(1+s)^{-1/2\alpha}ds\Big)\\
    &\leq ct^{-1/2\alpha-j/\alpha}
    \end{align*}
    and $II$ is estimated by
    \begin{align*}\int_{t/2}^te^{-(t-s)}\|\partial_x^{j+1}(F^n(s)^2\|_{L^2}ds
    &\leq c\int_{t/2}^te^{-(t-s)}\sum_{k=0}^{j+1}\|\partial_x^kF^n(s)\|_{L^2}\|\partial_x^{j+1-k}F^n(s)\|_{L^\infty}ds\\
    &\leq c\int_{t/2}^te^{-(t-s)}\sum_{k=0}^{j+1}\|\partial_x^kF^n(s)\|_{L^2}\|\partial_x^{j+1-k}F^n(s)\|_{L^2}^{1/2}\\
    &\quad\times \|\partial_x^{j+2-k}F^n(s)\|_{L^2}^{1/2}ds\\ &\leq c\int_{t/2}^te^{-(t-s)}(1+s)^{-5/2\alpha-j/\alpha}ds\\
    &\leq ct^{-5/2\alpha-j/\alpha} \leq ct^{-1/2\alpha-j/\alpha}.
    \end{align*}
    We have showed that $\|F^{n+1}(t)\|_{\dot{H}^j}\leq c(1+t)^{-1/2\alpha-j/\alpha}$ and by induction, this estimate becomes true for any $n\geq 0$.

 \noindent
{\bf Step 2.} We claim that if for all $j\geq 0$,
$\|u(t)-F^n(t)\|_{\dot{H}^j}\leq c(1+t)^{-r_j(n)}$ and
$r_j(n)=\frac j\alpha+r_0(n)$,
then $$\|u(t)-F^{n+1}(t)\|_{\dot{H}^j}\leq c\left\{\begin{array}{lll}(1+t)^{-1/2\alpha-j/\alpha-1/\alpha} &\rm{if} & 1-\frac{1}{2\alpha}-r_0(n)<0,\\
(1+t)^{-1/2\alpha-j/\alpha-1/\alpha}\log(1+t) &\rm{if}& 1-\frac{1}{2\alpha}-r_0(n)=0,\\ (1+t)^{-1/2\alpha-j/\alpha-1/\alpha+1-1/2\alpha-r_0(n)}
&\rm{if} &1-\frac{1}{2\alpha}-r_0(n)>0.\end{array}\right.$$
    Indeed, first for $t\leq 1$ it is clear that $\|u(t)-F^{n+1}(t)\|_{\dot{H}^j}$ is bounded. If $t>1$ we have by definition of $F^n$,
    \begin{align*}\|u(t)-F^{n+1}(t)\|_{\dot{H}^j} &\leq \frac 12\int_0^t\|\partial_x^{j+1}S_\alpha(t-s)\ast[u^2(s)-(F^n(s))^2]\|_{L^2}ds\\
    &= \int_0^{t/2}\ldots ds+\int_{t/2}^t\ldots ds :=III+IV.
    \end{align*}
    We bound the contribution of $III$ by
    \begin{align*}III &\leq c\int_0^{t/2}\|\partial_x^{j+1}S_\alpha(t-s)\|_{L^2}\|u^2(s)-(F^n(s))^2\|_{L^1}ds\\
    &\leq c\int_0^{t/2}(t-s)^{-1/2\alpha-(j+1)/\alpha}\|u(s)-F^n(s)\|_{L^2}(\|u(s)\|_{L^2}+\|F^n(s)\|_{L^2})ds\\
    &\leq c\int_0^{t/2}(t-s)^{-1/2\alpha-(j+1)/\alpha}(1+s)^{-1/2\alpha-r_0(n)}ds\\
    &\leq c\left\{\begin{array}{lll}(1+t)^{-1/2\alpha-j/\alpha-1/\alpha} &\rm{if} & 1-\frac{1}{2\alpha}-r_0(n)<0,\\
    (1+t)^{-1/2\alpha-j/\alpha-1/\alpha}\log(1+t) &\rm{if}& 1-\frac{1}{2\alpha}-r_0(n)=0,\\
    (1+t)^{-1/2\alpha-j/\alpha-1/\alpha+1-1/2\alpha-r_0(n)} &\rm{if} &1-\frac{1}{2\alpha}-r_0(n)>0.\end{array}\right.
    \end{align*}
    Then we decompose $IV$ as
    \begin{align*}IV &= c\int_{t/2}^t\Big(\intr |\xi|^{2(j+1)}e^{-2(t-s)|\xi|^\alpha}|\mathcal{F}[u^2(s)-(F^n(s))^2](\xi)|^2d\xi\Big)^{1/2}ds\\
    &\leq c\int_{t/2}^t\Big(\int_{|\xi|<1}\ldots d\xi\Big)^{1/2}ds+\int_{t/2}^t\Big(\int_{|\xi|>1}\ldots d\xi\Big)^{1/2}ds\\ &:= IV_1+IV_2.
    \end{align*}
    Low frequencies are treated as follows,
    \begin{align}\label{est4.1}\notag IV_1 &\leq \int_{t/2}^t\|\partial_x^{j+1}S_\alpha(1+t-s)\ast[u^2(s)-(F^n(s))^2]\|_{L^2}ds\\\notag
    &\leq c\int_{t/2}^t\|S_\alpha(1+t-s)\|_{L^2}\|\partial_x^{j+1}[u^2(s)-(F^n(s))^2]\|_{L^1}ds\\
    \notag &\leq c\int_{t/2}^t(1+t-s)^{-1/2\alpha}\sum_{k=0}^{j+1}\|\partial_x^k[u(s)-F^n(s)]\|_{L^2}(\|\partial_x^{j+1-k}u(s)\|_{L^2}+
    \|\partial_x^{j+1-k}F^n(s)\|_{L^2})ds\\\notag &\leq c\int_{t/2}^t(1+t-s)^{-1/2\alpha}\sum_{k=0}^{j+1}(1+s)^{-r_k(n)-1/2\alpha-(j+1-k)/\alpha}ds\\
    &\leq c\sum_{k=0}^{j+1}t^{-r_k(n)+k/\alpha-j/\alpha+1-2/\alpha}
    \end{align}
    and since $r_k(n)=\frac k\alpha+r_0(n)$, we infer $IV_1\leq ct^{-r_0(n)-j/\alpha+1-2/\alpha}.$ In the same way,
    \begin{align}\label{est4.2}\notag IV_2 &\leq c\int_{t/2}^te^{-(t-s)}\|\partial_x^{j+1}[u^2(s)-(F^n(s))^2]\|_{L^2}ds\\ \notag
     &\leq c\int_{t/2}^te^{-(t-s)}\sum_{k=0}^{j+1}(1+s)^{-r_k(n)-1/\alpha-(j+1-k)/\alpha}ds \\
     \notag &\leq c\sum_{k=0}^{j+1}t^{-r_k(n)+k/\alpha-j/\alpha-2/\alpha}\\  &\leq ct^{-r_0(n)-j/\alpha-2/\alpha}.
    \end{align}
    Combining (\ref{est4.1}) and (\ref{est4.2}), we deduce
    $$IV\leq  c\left\{\begin{array}{lll}(1+t)^{-1/2\alpha-j/\alpha-1/\alpha} &\rm{if} & 1-\frac{1}{2\alpha}-r_0(n)<0,\\
    (1+t)^{-1/2\alpha-j/\alpha-1/\alpha}\log(1+t) &\rm{if}& 1-\frac{1}{2\alpha}-r_0(n)=0,\\
    (1+t)^{-1/2\alpha-j/\alpha-1/\alpha+1-1/2\alpha-r_0(n)} &\rm{if} &1-\frac{1}{2\alpha}-r_0(n)>0.\end{array}\right.$$

\noindent
{\bf Step 3.} Construction of $r_j(n)$ and conclusion.\\
    We define the sequence $r_j(n)$ by iteration. Set $r_j(0)=\frac1{2\alpha}+\frac j\alpha+\frac 2\alpha-1$ for all $j\geq
    0$.
    We have $\|u(t)-F^0(t)\|_{\dot{H}^j}\leq c(1+t)^{-r_j(0)}$ by Theorem \ref{th-ordre1}. If $r_j(n)$ is constructed for all $j$,
    then we set
    \begin{equation}\label{defrjn}r_j(n+1)=\left\{\begin{array}{lll}\frac{1}{2\alpha}+\frac{j}{\alpha}+\frac{1}{\alpha}
    &\rm{if}& 1-\frac{1}{2\alpha}-r_0(n)\leq 0,\\ r_0(n)+\frac{j}{\alpha}+\frac{2}{\alpha}-1 &\rm{if} & 1-
    \frac{1}{2\alpha}-r_0(n)>0.\end{array}\right.\end{equation}
    We easily see that $r_j(n)=\frac j\alpha+r_0(n)$ for all $j$, thus Step 2 shows that for any $n\geq 0$ satisfying $1-\frac{1}{2\alpha}-r_0(n)\leq 0$,
    \begin{multline}\label{estuF}\|u(t)-F^{n+1}(t)\|_{\dot{H}^j}\\\leq c\left\{\begin{array}{lll}(1+t)^{-1/2\alpha-j/\alpha-1/\alpha}
    &\rm{if} & 1-\frac{1}{2\alpha}-r_0(n)<0,\\ (1+t)^{-1/2\alpha-j/\alpha-1/\alpha}\log(1+t) &\rm{if} & 1-\frac{1}{2\alpha}-r_0(n)=0.
    \end{array}\right.\end{multline}
    Let us prove that the sequence $n\mapsto r_j(n)$ is eventually constant. Suppose that $1-\frac{1}{2\alpha}-r_0(n)>0$ for all
    $n\geq 0$. Then by (\ref{defrjn}) we obtain $r_j(n+1)=r_0(n)+\frac j\alpha+\frac 2\alpha-1$ ($\forall n$). In particular $r_0(n+1)=r_0(n)
    +\frac 2\alpha-1$ and thus $r_0(n)=n(\frac 2\alpha-1)+r_0(0)=(n+1)(\frac 2\alpha-1)+\frac 1{2\alpha}$. Since $\frac 2\alpha-1>0$,
    this contradicts the assumption $r_0(n)<1-\frac 1{2\alpha}$ for $n$ large enough. Hence there exists $n\geq 0$ such that $1-\frac 1{2\alpha}-r_0(n)\leq 0$
    and we can set $$N=\min\big\{n\geq 0 : 1-\frac 1{2\alpha}-r_0(n)\leq 0\big\}.$$ For this value of $N$, it is not too difficult to see
    that $$r_j(n)=\left\{\begin{array}{lll}(n+1)(\frac 2\alpha-1)+\frac 1{2\alpha}+\frac j\alpha &\rm{if} & n\leq N,\\ \frac 1{2\alpha}+\frac
    j\alpha+\frac 1\alpha &\rm{if} & n>N.\end{array}\right.$$
    It follows that $N=\min\{n\geq 0 : 1-\frac 1\alpha-(n+1)(\frac 2\alpha-1)\leq 0\}=\min\{n\geq 0 : \alpha\leq \frac{2n+3}{n+2}\}$.
    From this and (\ref{estuF}) we infer
    $$\|u(t)-F^{N+1}(t)\|_{\dot{H}^j}\leq c\left\{\begin{array}{lll}(1+t)^{-1/2\alpha-j/\alpha-1/\alpha} &\rm{if} & \alpha<\frac{2N+3}{N+2},
    \\ (1+t)^{-1/2\alpha-j/\alpha-1/\alpha}\log(1+t) &\rm{if} & \alpha=\frac{2N+3}{N+2}
    .\end{array}\right.$$
\end{proof}

\section*{Acknowledgment}
The author thanks Francis Ribaud for several encouragements and helpful discussions.

\bibliographystyle{plain}
\bibliography{ref}

\end{document}